\documentclass[10pt]{amsart}
\usepackage{amsfonts}
\usepackage{ifthen}
\usepackage{amsthm}
\usepackage{amsmath}
\usepackage{graphicx}
\usepackage{amscd,amssymb,amsthm}
\usepackage{graphicx}

\newcounter{minutes}
\divide\time by 60
\newcounter{hours}
\multiply\time by 60 \addtocounter{minutes}{-\time}

\title[Bounds for Tur\'anians of modified Bessel functions]{Bounds for Tur\'anians of modified Bessel functions$^{\bigstar}$}

\thanks{$^{\bigstar}$Research supported by Romanian National Research Council, project number PN-II-RU-TE\underline{ }190/2013.}
\author[\'Arp\'ad Baricz]{\'Arp\'ad Baricz}
\address{Department of Economics, Babe\c{s}-Bolyai University, 400591 Cluj-Napoca, Romania} \email{bariczocsi@yahoo.com}

\newtheorem{theorem}{Theorem}
\newtheorem{corollary}{Corollary}

\keywords{Modified Bessel functions of the first and second kind,
Bessel functions of the first and second kind, product of modified
Bessel functions, sharp bounds, Tur\'an-type inequalities,
functional inequalities.} \subjclass[2010]{33C10, 39B62}
\begin{document}

\def\thefootnote{}
\footnotetext{ \texttt{File:~\jobname .tex,
          printed: \number\year-0\number\month-\number\day,
          \thehours.\ifnum\theminutes<10{0}\fi\theminutes}
} \makeatletter\def\thefootnote{\@arabic\c@footnote}\makeatother

\maketitle

\begin{center}
{\footnotesize{\textsf{Dedicated to Bor\'oka and Kopp\'any}}}
\end{center}

%=======================================================================================================================================================

\begin{abstract}
Motivated by some applications in applied mathematics, biology,
chemistry, physics and engineering sciences, new tight Tur\'an type
inequalities for modified Bessel functions of the first and second
kind are deduced. These inequalities provide sharp lower and upper
bounds for the Tur\'anian of modified Bessel functions of the first
and second kind, and in most cases the relative errors of the bounds
tend to zero as the argument tends to infinity. The chief tools in
our proofs are some ideas of Gronwall \cite{gronwall} on ordinary differential equations, an integral
representation of Ismail \cite{ismail,ismail2} for the quotient of
modified Bessel functions of the second kind and some results of Hartman and
Watson \cite{hartman,swatson,swatson2}. As applications of the main results some sharp
Tur\'an type inequalities are presented for the product of modified
Bessel functions of the first and second kind and it is shown that
this product is strictly geometrically concave.
\end{abstract}

\section{\bf Introduction}

Let us denote by $I_{\nu}$ and $K_{\nu}$ the modified Bessel
functions of the first and second kind of real order $\nu,$ which
are the linearly independent particular solutions of the second
order modified Bessel differential equation. For definitions,
recurrence formulas and many important properties of modified Bessel
functions of the first and second kind we refer to the classical
book of Watson \cite{watson}. Recall that the modified Bessel
function $I_{\nu},$ called also sometimes as the Bessel function of
the first kind with imaginary argument, has the series
representation \cite[p. 77]{watson}
$$I_{\nu}(x)=\sum_{n\geq0}\frac{(x/2)^{2n+\nu}}{n!\Gamma(n+\nu+1)},$$
where $\nu\neq-1,-2,\dots$ and $x\in\mathbb{R}.$ The modified Bessel
function of the second kind $K_{\nu},$ called also sometimes as the
MacDonald or Hankel function, is defined as \cite[p. 78]{watson}
$$K_{\nu}(x)=\frac{\pi}{2}\frac{I_{-\nu}(x)-I_{\nu}(x)}{\sin\nu\pi},$$
where the right-hand side of this equation is replaced by its
limiting value if $\nu$ is an integer or zero. We note that in view
of the above series representation $I_{\nu}(x)>0$ for all $\nu>-1$
and $x>0.$ Similarly, by using the familiar integral representation
\cite[p. 181]{watson} $$K_{\nu}(x)=\int_0^{\infty}e^{-x\cosh
t}\cosh(\nu t)dt,$$ which holds for each $x>0$ and
$\nu\in\mathbb{R},$ one can see that $K_{\nu}(x)>0$ for all $x>0$
and $\nu\in\mathbb{R}.$ These functions are among the most important
functions of the mathematical physics and have been used (for
example) in problems of electrical engineering, hydrodynamics,
acoustics, biophysics, radio physics, atomic and nuclear physics,
information theory. These functions are also an effective tool for
problem solving in areas of wave mechanics and elasticity theory.
Modified Bessel functions of the first and second kind are an
inexhaustible subject, there are always more useful properties than
one knows. Recently, there has been a vivid interest on bounds for
ratios of modified Bessel functions and on Tur\'an type inequalities
for these functions. For more details we refer the interested reader
to the most recent papers in the subject
\cite{baPAMS,baBAMS,baPEMS,bapo,koko,laforgia,natalini,segura} and
to the references therein. It is important to mention here that surprisingly the existing
Tur\'an type inequalities for modified Bessel functions of the first and second kind
appear in many problems of applied mathematics, biology,
chemistry, physics and engineering sciences, as we can see in Sections 2 and 3. Motivated by the above applications,
in this paper our aim is to reconsider the Tur\'an type inequalities for modified Bessel
functions of the first and second kind. By using some ideas of
Gronwall \cite{gronwall} on ordinary differential equations, an integral representation of Ismail
\cite{ismail,ismail2} for the quotient of modified Bessel functions
of the second kind, results of Hartman and Watson
\cite{hartman,swatson,swatson2} and some recent results of Segura
\cite{segura}, in the present paper we make a contribution to the
subject and we deduce some new tight Tur\'an type inequalities for
modified Bessel functions of the first and second kind. These
inequalities, studied in details in Sections 3 and 4, provide sharp
lower and upper bounds for the Tur\'anian of modified Bessel
functions of the first and second kind, and in most cases the
relative errors of the bounds tend to zero as the argument tends to
infinity. In addition, in Section 3 we point out some mathematical
errors in the papers of Gronwall \cite{gronwall}, Hamsici and
Martinez \cite{martinez} and of Joshi and Bissu \cite{joshi}, and we
also correct these errors. Moreover, we present new proofs for the
right-hand sides of \eqref{turan1} and \eqref{turan2}, and also for
some of the results of Hartman and Watson \cite{swatson}. At the end
of Section 3 an open problem is discussed in details, which may be
of interest for further research. Finally, in Section 5 we present
some applications of the main results of Section 3 and 4. Here we
prove that the product of modified Bessel functions of the first and
second kind is strictly geometrically concave and we deduce some
sharp Tur\'an type inequalities for this product.

\section{\bf Tur\'an type inequalities for modified Bessel functions}
\setcounter{equation}{0}

In this section our aim is to recall the existing Tur\'an type inequalities for modified Bessel functions of the first and second kind, and to survey the problems in which these inequalities appear. First we focus on the following Tur\'an-type inequalities, which
hold for all $\nu>-1$ and $x>0$
\begin{equation}\label{turan1}0< I_{\nu}^2(x)-I_{\nu-1}(x)I_{\nu+1}(x)
< \frac{1}{\nu+1}\cdot I_{\nu}^2(x).\end{equation} Note that their
analogue hold for all $|\nu|>1$ and $x>0$
\begin{equation}\label{turan2}
\frac{1}{1-|\nu|}\cdot
K_{\nu}^2(x)<K_{\nu}^2(x)-K_{\nu-1}(x)K_{\nu+1}(x)<0.
\end{equation}
These inequalities have attracted the interest of many
mathematicians, and were rediscovered by many times by different
authors in different forms. To the best of author's knowledge the
Tur\'an type inequality \eqref{turan1} for $\nu>-1$ was proved first
by Thiruvenkatachar and Nanjundiah \cite{tiru}. The left-hand side
was proved also later by Amos \cite[p. 243]{amos} for $\nu\geq0.$
Joshi and Bissu \cite{joshi} proved also the left-hand side of
\eqref{turan1} for $\nu\geq0,$ while Lorch \cite{lorch} proved that
this inequality holds for all $\nu\geq -1/2.$ Recently, the author
\cite{baBAMS} reconsidered the proof of Joshi and Bissu \cite{joshi}
and pointed out that \eqref{turan1} holds true for all $\nu>-1$ and
the constants $0$ and $1/(\nu+1)$ in \eqref{turan1} are best
possible. Note that, as it was shown in \cite{baricz4,lorch}, the
function $\nu\mapsto I_{\nu+\alpha}(x)/I_{\nu}(x)$ is decreasing for
each fixed $\alpha\in(0,2]$ and $x>0,$ where $\nu>-1$ and $\nu\geq
-(\alpha+1)/2.$ Consequently, the function $\nu\mapsto I_{\nu}(x)$
is log-concave on $(-1,\infty)$ for each fixed $x>0,$ as it was
pointed out in \cite{baricz4}. See also the paper of Segura
\cite{segura} for an alternative proof of \eqref{turan1}. For the
sake of completeness it should be also mentioned here that the
right-hand side of \eqref{turan2} was first proved independently by
Ismail and Muldoon \cite{muldoon} and van Haeringen \cite{Har}, and
rediscovered later by Laforgia and Natalini \cite{natalini}. Note
that in \cite{muldoon} the authors actually proved that for all
fixed $x>0$ and $\beta>0,$ the function $\nu\mapsto
K_{\nu+\beta}(x)/K_{\nu}(x)$ is increasing on $\mathbb{R}.$ Another
proof of the right-hand side of \eqref{turan2}, which holds true for
all $\nu\in\mathbb{R}$, was given in \cite{baricz4}. Recently,
Baricz \cite{baBAMS} and Segura \cite{segura}, proved the two sided
inequality in \eqref{turan2} by using different approaches. Note
that in \cite{baBAMS} the inequality \eqref{turan2} is stated only
for $\nu>1,$ however, because of the well-known symmetry relation
$K_{\nu}(x)=K_{-\nu}(x)$ we can change $\nu$ by $-\nu.$ See also
\cite{bapo} for more details on \eqref{turan2}.

It is also worth to mention that according to the corresponding
recurrence relations for the modified Bessel functions of the first
and second kind (see \cite[p. 251]{nist} or \cite[p. 79]{watson}), the left-hand side of \eqref{turan1} is equivalent
to
\begin{equation}\label{turan3}\frac{xI_{\nu}'(x)}{I_{\nu}(x)}< \sqrt{x^2+\nu^2},\end{equation}
while the right-hand side of \eqref{turan2} is equivalent to
\begin{equation}\label{turan4}\frac{xK_{\nu}'(x)}{K_{\nu}(x)}<
-\sqrt{x^2+\nu^2}.\end{equation} Moreover, the inequalities
\eqref{turan3} and \eqref{turan4} together imply that the function
$x\mapsto P_{\nu}(x)=I_{\nu}(x)K_{\nu}(x)$ is strictly decreasing on
$(0,\infty)$ for all $\nu>-1.$ See \cite{baPAMS,baBAMS} for more
details. Note that the above monotonicity property of $P_{\nu}$ was
proved earlier by Penfold et al. \cite[p. 142]{penfold} by using a
different approach. The study in \cite{penfold} was motivated by a
problem in biophysics. See also the paper of Grandison et al.
\cite{grandison} for more details. For the sake of completeness we
recall also that the inequality \eqref{turan3} was deduced
first\footnote{To prove \eqref{turan3} Gronwall \cite[p.
277]{gronwall} claimed that the function $x\mapsto
\sqrt{x^2+\nu^2}-{xI_{\nu}'(x)}/{I_{\nu}(x)}$ is increasing on
$(0,\infty)$ for all $\nu>0.$ As it will be pointed out in the next
section the above claim is not true. All the same, the inequality
\eqref{turan3} is valid, and in view of \eqref{ynu} it follows from
the fact \cite[p. 277]{gronwall} that the function $x\mapsto
{xI_{\nu}'(x)}/{I_{\nu}(x)}-\nu$ is increasing on $(0,\infty)$ for
all $\nu>0.$} by Gronwall \cite[p. 277]{gronwall} for $\nu>0,$
motivated by a problem in wave mechanics. This inequality was
deduced also for $\nu\in\{1,2,\dots\}$ by Phillips and Malin
\cite[p. 407]{phillips}, and for $\nu>0$ by Amos \cite[p. 241]{amos}
and Paltsev \cite[eq. (21)]{paltsev}. The inequality \eqref{turan4}
was deduced first for $\nu\in\{1,2,\dots\}$ by Phillips and Malin
\cite[p. 407]{phillips}, and later for $\nu\geq0$ by Paltsev
\cite[eq. (22)]{paltsev}. We note that the Tur\'an type inequalities
\eqref{turan1}, \eqref{turan2}, \eqref{turan3} and \eqref{turan4} as
well as the monotonicity of the product of $P_{\nu}$ were used in
various problems related to modified Bessel functions in various
topics of applied mathematics, biology, chemistry and physics. For
reader's convenience we list here some of the related things:

\begin{enumerate}

\item[\bf 1.] The monotonicity of $P_{\nu}$ for $\nu>1$ is used
(without proof) in some papers about the hydrodynamic and
hydromagnetic instability of different cylindrical models. See for
example \cite{radwandi,radwanha}. See also the paper of Hasan
\cite{hasan}, where the electrogravitational instability of
onoscillating streaming fluid cylinder under the action of the
selfgravitating, capillary and electrodynamic forces has been
discussed. In these papers the authors use (without proof) the
inequality $$P_{\nu}(x)<\frac{1}{2}$$ for all $\nu>1$ and $x>0.$ We
note that the above inequality readily follows from the fact that
$P_{\nu}$ is strictly decreasing on $(0,\infty)$ for all $\nu>-1.$
More precisely, for all $x>0$ and $\nu>1$ we have
$$P_{\nu}(x)<\lim_{x\to0}P_{\nu}(x)=\frac{1}{2\nu}<\frac{1}{2}.$$

\item[\bf 2.] The Tur\'an type inequality \eqref{turan1} and the
right-hand side of \eqref{turan2}, together with the monotonicity of
$P_{\nu}$ were used, among other things, by Klimek and McBride
\cite{klimek} to prove that a Dirac operator, subject to
Atiyah-Patodi-Singer-like boundary conditions on the solid torus,
has a bounded inverse, which is actually a compact operator.

\item[\bf 3.] Recently, Simitev and Biktashev \cite{simitev} used the fact
that the function $x\mapsto xI_{\nu}'(x)/I_{\nu}(x)$ is increasing
on $(0,\infty)$ together with the inequality \eqref{turan3} in the
study of asymptotic restitution curves in the caricature Noble model
of electrical excitation in the heart. As it was pointed out above
the inequality \eqref{turan3} is equivalent to the left-hand side of
the Tur\'an type inequality \eqref{turan1}. Moreover, because of the
relation \cite[p. 339]{joshi}, \cite[p. 256]{baBAMS}
\begin{equation}\label{deltaI}x\left[I_{\nu}^2(x)-I_{\nu-1}(x)I_{\nu+1}(x)\right]=
I_{\nu}^2(x)\left[\frac{xI_{\nu}'(x)}{I_{\nu}(x)}\right]',\end{equation}
the fact that the function $x\mapsto xI_{\nu}'(x)/I_{\nu}(x)$ is
increasing is also equivalent to the left-hand side of the Tur\'an
type inequality \eqref{turan1}. Thus, Simitev and Biktashev
\cite{simitev} actually used two times in their study exactly the
left-hand side of the Tur\'an type inequality \eqref{turan1}. Here
it is important to note that very recently, in order to prove
that\footnote{An alternative proof for this result is as follows:
according to Watson \cite{swatson2} the function $x\mapsto
I_{\nu+1}(x)/I_{\nu}(x)$ is increasing on $(0,\infty)$ for all
$\nu\geq-1/2.$ Thus,
$\left[xI_{\nu+1}(x)/I_{\nu}(x)\right]'=I_{\nu+1}(x)/I_{\nu}(x)+x\left[I_{\nu+1}(x)/I_{\nu}(x)\right]'>0$
for all $\nu\geq-1/2$ and $x>0.$}
$\left[xI_{\nu+1}(x)/I_{\nu}(x)\right]'>0$ for all $\nu\geq0$ and
$x>0,$ Schlenk and Sicbaldi \cite[p. 622]{sicbaldi} rediscovered the
left-hand side of the inequality \eqref{turan1}. They used the
relation
$$\left[\frac{xI_{\nu+1}(x)}{I_{\nu}(x)}\right]'=\frac{x\left[I_{\nu}^2(x)-I_{\nu-1}(x)I_{\nu+1}(x)\right]}{I_{\nu}^2(x)},$$
which in view of the recurrence relation $xI_{\nu}'(x)=\nu
I_{\nu}(x)+xI_{\nu+1}(x),$ is actually the same as \eqref{deltaI}.
We also mention that in \cite{sicbaldi} the authors rediscovered
also the corresponding Tur\'an type inequality for Bessel functions
of the first kind. These results on Bessel and modified Bessel
functions of the first kind were used in \cite{sicbaldi} to study bifurcating extremal domains for the first eigenvalue
of the Laplacian. More precisely, in \cite{sicbaldi} the Tur\'an type inequalities
for Bessel and modified Bessel functions were used in the study of
the monotonicity of the first eigenvalue of a linearized operator,
in order to show that this operator satisfies the assumptions of the
Crandall-Rabinowitz theorem, implying the main result of
\cite{sicbaldi}. Finally, for a survey on the Tur\'an type
inequalities for Bessel functions of the first kind the interested
reader is referred to \cite{bpog}.

\item[\bf 4.] Note that, as it was pointed out in \cite{segura}, the
left-hand side of the Tur\'an type inequality \eqref{turan2}
provides actually an upper bound for the effective variance of the
generalized Gaussian distribution. More precisely, Alexandrov and
Lacis \cite{lacis}
 used (without proof) the inequality
$0<v_{\rm{eff}}<1/(\mu-1)$ for $\mu=\nu+4,$ where
$$v_{\rm{eff}}:=\frac{\displaystyle\left[\int_0^{\infty}r^2f_{\nu}(r)dr\right]
\left[\int_0^{\infty}r^4f_{\nu}(r)dr\right]}
{\displaystyle\left[\int_0^{\infty}r^3f_{\nu}(r)dr\right]^2}-1=
\frac{K_{\mu-1}(1/w)K_{\mu+1}(1/w)}{\left[K_{\mu}(1/w)\right]^2}-1$$
is the effective variance of the generalized Gaussian distribution
and
$$f_{\nu}(r)=\frac{1}{2K_{\nu+1}(1/w)}\frac{r^{\nu}}{s^{\nu+1}}
\exp\left[-\frac{1}{2w}\left(\frac{s}{r}+\frac{r}{s}\right)\right]$$
is the generalized inverse Gaussian particle size distribution
function, $w$ represents the width of the distribution, $s$ is an
effective size parameter, and $\nu$ is the order of the
distribution.

\item[\bf 5.] Simon \cite{simon} used the Tur\'an type inequality
\eqref{turan4} for $\nu=1/3$ to prove that the positive $1/3-$stable
distribution with density
$$f_{1/3}(x)=\frac{1}{3\pi x^{3/2}}K_{1/3}\left(\frac{2}{3\sqrt{3x}}\right)$$
is multiplicative strongly unimodal in the sense of
Cuculescu-Theodorescu, that is, $t\mapsto f_{1/3}(e^t)$ is
log-concave in $\mathbb{R}.$ Here for $\alpha\in(0,1)$ the positive
$\alpha-$stable density is normalized such that
$$\int_0^{\infty}e^{-\lambda t}f_{\alpha}(t)dt=
E\left[e^{-\lambda Z_{\alpha}}\right]=e^{-\lambda^{\alpha}},$$ where
$\lambda\geq0$ and $Z_{\alpha}$ is the corresponding random
variable.

\item[\bf 6.] Recently, motivated by some results in finite
elasticity, Laforgia and Natalini \cite{natalini} proved that for
$x>0$ and $\nu\geq0$ the following inequality is valid
\begin{equation}\label{turan5}\frac{I_{\nu}(x)}{I_{\nu-1}(x)}
>\frac{-\nu+\sqrt{x^2+\nu^2}}{x}.\end{equation}
We note an alternative proof of \eqref{turan5} was given recently by
Kokologiannaki \cite[eq. (2.1)]{koko}. Moreover, as it was pointed
out in \cite{bapo}, \eqref{turan5} was proved already by Amos
\cite[eq. (9)]{amos} for $\nu\geq 1$ and $x>0.$ It is also worth to
mention that the authors showed in \cite{bapo} that the inequality
\eqref{turan5} is equivalent to \eqref{turan3}, which is equivalent
to the left-hand side of \eqref{turan1}. Observe that the inequality
\eqref{turan5} can be rewritten in the form
\begin{equation}\label{turan6}\frac{1}{x}\frac{I_{\nu}(x)}{I_{\nu-1}(x)}
>\frac{1}{\nu+\sqrt{x^2+\nu^2}},\end{equation}
where $x>0$ and $\nu\geq0.$ In \cite{segura} it was pointed out that
the inequality \eqref{turan6}, which is actually equivalent to the
left-hand side of \eqref{turan1}, appears in a problem of chemistry.
More precisely, in \cite{ford} the authors considered the mean
number of molecules of a given class dissolved in a water droplet
and compared the so-called classical and stochastic approaches. If
$n_c$ and $n_s$ are the respective mean numbers of molecules by
using the classical and stochastic approaches, then according to
Segura \cite{segura}, after the redefinition of the variables it can
be shown that
$$n_c=\frac{x^2}{4}\frac{1}{\nu+1+\sqrt{x^2+(\nu+1)^2}}\ \
\mbox{and}\ \ n_s=\frac{x}{4}\frac{I_{\nu+1}(x)}{I_{\nu}(x)}$$ and
by using the inequality \eqref{turan6} for all $x>0$ and $\nu\geq-1$
we have $n_s>n_c.$ Note that this inequality was known before only
for small or large values of $x.$

\item[\bf 7.] The analogue of \eqref{turan5} for modified Bessel functions of the second kind, that is,
\begin{equation}\label{turan5'}\frac{K_{\nu}(x)}{K_{\nu-1}(x)}
<\frac{\nu+\sqrt{x^2+\nu^2}}{x}\end{equation} was proved recently by
Laforgia and Natalini \cite{natalini} for $x>0$ and
$\nu\in\mathbb{R}.$ Note that in \cite{bapo} the authors pointed out
that in fact \eqref{turan5'} is equivalent to \eqref{turan4}, which
is equivalent to the right-hand side of \eqref{turan2}. The
inequality \eqref{turan5'} was used recently by Fabrizi and
Trivisano \cite{fabrizi} to deduce an upper bound for the expected
value of a random variable which has a generalized inverse Gaussian
distribution, while Lechleiter and Nguyen \cite{nguyen} used the
inequality \eqref{turan5'} to deduce an error estimate for an
approximation to the waveguide Green's function.

\item[\bf 8.] It is also interesting to note that the Tur\'anian
$K_{\nu}^2(x)-K_{\nu-1}(x)K_{\nu+1}(x)$ appears in the variance of
the non-central $F-$Bessel distribution defined by Thabane and
Drekic \cite{thabane}, and it appears also in \cite[eq.
(37)]{bhatta}, related with the variance of a different
distribution. Moreover, in \cite{bpv} the authors investigated the
convexity with respect to power means of the modified Bessel
functions $I_{\nu}$ and $K_{\nu}$ by using the Tur\'an type
inequalities presented above. We note that the left-hand side of the
inequality \eqref{turan1} was used also by Milenkovic and Compton
\cite{milenkovic}, and for $\nu=1$ by Bertini et al. \cite{bertini}.
The property that $x\mapsto xI_{\nu}'(x)/I_{\nu}(x)$ is increasing
on $(0,\infty)$ for all $\nu\geq0$ was used by Giorgi and Smits
\cite[p. 237]{giorgi}, \cite[p. 610]{giorgi2}, and also by Lombardo
et al. \cite{lombardo} together with its analogue that $x\mapsto
xK_{\nu}'(x)/K_{\nu}(x)$ is decreasing on $(0,\infty)$ for all
$\nu\geq0.$ The later property is used in \cite{lombardo} without
proof, however, this is actually equivalent to the right-hand side
of the Tur\'an type inequality \eqref{turan2}, according to relation
\cite[p. 259]{baBAMS}
\begin{equation}\label{deltaK}x\left[K_{\nu}^2(x)-K_{\nu-1}(x)K_{\nu+1}(x)\right]=
K_{\nu}^2(x)\left[\frac{xK_{\nu}'(x)}{K_{\nu}(x)}\right]'.\end{equation}
\end{enumerate}

\section{\bf Tur\'an type inequalities for modified Bessel functions
of the first kind} \setcounter{equation}{0}

In this section our aim is to study the Tur\'an type inequalities
for modified Bessel functions of the first kind motivated by the
above applications and by the paper of Hamsici and Martinez
\cite{martinez}. Joshi and Bissu \cite{joshi} proved for $x>0$ and
$\nu\geq0$ the following two Tur\'an type inequalities
$$I_{\nu}^2(x)-I_{\nu-1}(x)I_{\nu+1}(x)<\frac{4}{j_{\nu,1}^2}\cdot I_{\nu}^2(x)$$
and
\begin{equation}\label{turan7}I_{\nu}^2(x)-I_{\nu-1}(x)I_{\nu+1}(x)
<\frac{1}{x+\nu}\cdot I_{\nu}^2(x),\end{equation} where $j_{\nu,1}$
is the first positive zero of the Bessel function $J_{\nu}.$ Observe
that, in view of the Rayleigh inequality \cite[p. 502]{watson}
$j_{\nu,1}^2>4(\nu+1),$ the first inequality would be an improvement
of the right-hand side of \eqref{turan1}. However, based on
numerical experiments, unfortunately both of the above inequalities
from \cite{joshi} are not valid for all $x>0$ and $\nu\geq0.$ The
reason for that the first inequality is not true for all $x>0$ and
$\nu\geq0$ is that in the right-hand side of \eqref{turan1} the
constant $1/(\nu+1)$ is best possible, according to \cite[p.
257]{baBAMS}, and consequently cannot be improved by other constant
(independent of $x$). On the other hand, the inequality
\eqref{turan7} is not valid because its proof is not correct. By
using only the so-called Nasell inequality \cite[p. 253]{nasell}
$$1+\frac{\nu}{x}<\frac{I_{\nu}(x)}{I_{\nu+1}(x)}$$
it is not possible to prove the inequality \eqref{turan7}. Because
of this, the proofs of the extensions of \eqref{turan7} in \cite[p.
340]{joshi} cannot be correct too. We note that actually by using
some recent results of Segura \cite{segura} the inequality
\eqref{turan7} can be corrected. More precisely, let us focus on the
inequalities \cite[eqs. (45), (54)]{segura}
\begin{equation}\label{turan8}
\frac{1}{\nu+\frac{1}{2}+\sqrt{x^2+\left(\nu+\frac{1}{2}\right)^2}}\cdot
I_{\nu}^2(x)< I_{\nu}^2(x)-I_{\nu-1}(x)I_{\nu+1}(x)<
\frac{2}{\nu+1+\sqrt{x^2+(\nu+1)^2}}\cdot I_{\nu}^2(x),
\end{equation}
where\footnote{We note that in \cite[Theorem 9]{segura} and in its
proof it is assumed that $\nu\geq0.$ However, by using \cite[eq.
(21)]{segura}, in inequality \cite[eq. (53)]{segura} and in the
right-hand side of \cite[eq. (54)]{segura} we can assume that
$\nu>-1.$ Moreover, the left-hand side of the inequality \cite[eq.
(54)]{segura} is valid for all $\nu\geq-1.$} $x>0$ and $\nu\geq-1.$
By using the inequality \eqref{turan8} clearly we have
\begin{equation}\label{turan9}
\frac{1}{x+2\nu+1}\cdot I_{\nu}^2(x)<
I_{\nu}^2(x)-I_{\nu-1}(x)I_{\nu+1}(x)< \frac{2}{x+\nu+1}\cdot
I_{\nu}^2(x),
\end{equation}
with the same range of validity as in \eqref{turan8}. The right-hand
side of the above inequality actually implies that \eqref{turan7}
can be corrected as
\begin{equation}\label{turan10}I_{\nu}^2(x)-I_{\nu-1}(x)I_{\nu+1}(x)
<\frac{2}{x+\nu}\cdot I_{\nu}^2(x),\end{equation} where $x>0$ and
$\nu\geq0.$ Recall that, according to \cite[p. 257]{baBAMS}, for
$$\varphi_{\nu}(x)=1-\frac{I_{\nu-1}(x)I_{\nu+1}(x)}{I_{\nu}^2(x)}$$
we have $\lim_{x\to \infty}\varphi_{\nu}(x)=0$ and
$\lim_{x\to0}\varphi_{\nu}(x)={1}/(\nu+1).$ Thus, all the
inequalities in \eqref{turan8} and \eqref{turan9} are sharp as
$x\to\infty,$ while the right-hand side of \eqref{turan8} is also
sharp as $x\to0.$ Observe that clearly the left-hand sides of
\eqref{turan8} and \eqref{turan9} improve the left-hand side of
\eqref{turan1}, and the right-hand side of \eqref{turan8} improves
the right-hand side of \eqref{turan1} for all $x>0$ and $\nu>-1.$
The right-hand side of \eqref{turan9} also improves the right-hand
side of \eqref{turan1} for all $x\geq \nu+1>0.$ We note that in view
of the left-hand side of \eqref{turan8} it can be proved that the
inequality \eqref{turan7} is reversed for all $1/2\leq
x\leq\nu(\nu+1)$ and $\nu\geq0,$ which also shows that
\eqref{turan7} cannot be correct for all $x>0$ and $\nu\geq0.$

It is important to note here that recently Hamsici and Martinez
\cite[p. 1595]{martinez} used the Tur\'an type inequality
\eqref{turan7} and concluded that for all $x>0$ and $\nu>0$ we have
$$\hat{b}_2(x)=-\frac{I_{\nu}^2(x)}{x\left[I_{\nu}^2(x)-I_{\nu-1}(x)I_{\nu+1}(x)\right]}<-\frac{x+\nu}{x}<-1.$$
See also \cite[p. 70]{hamsici1} and \cite[p. 36]{hamsici2}. Since
the inequality \eqref{turan7} is not valid for all $x>0$ and
$\nu\geq0$ we can see that the left-hand side of the above
inequality is not valid too for all $x>0$ and $\nu>0.$ In view of
\eqref{turan10} the above inequality should be written as
$$\hat{b}_2(x)<-\frac{x+\nu}{2x}<-\frac{1}{2}.$$
This implies that the bias of the hyperplane in \cite[Proposition
4]{martinez} does not have the property that its absolute value is
greater than $1,$ at least according to the proof given in
\cite{martinez}. In view of the above correct inequality the
absolute value of the bias will be just greater than $1/2$ and this
means that proof of the assertion \cite[Proposition 4]{martinez}
``that the hyperplane given in (12) does not intersect with the
sphere and can be omitted for classification purposes'' is not
complete. All the same, by using the right-hand side irrational
bound in \eqref{turan8} we can prove that $\hat{b}_2(x)<-1,$ but
only for $0<x\leq4(\nu+1)/3$ and $\nu>-1.$ Moreover, by using a
result of Gronwall \cite{gronwall}, it is possible to show that the
claimed inequality $\hat{b}_2(x)<-1,$ that is,
\begin{equation}\label{turan11}
I_{\nu}^2(x)-I_{\nu-1}(x)I_{\nu+1}(x)<\frac{1}{x}\cdot I_{\nu}^2(x),
\end{equation}
is actually valid\footnote{We would like to mention here that Tanabe et al. \cite{tanabe} proposed an iterative algorithm by using fixed points to obtain the maximum likelihood estimate for one of the parameters of the $p$-variate von Mises--Fisher distribution on the $p$-dimensional unit hypersphere and for this they used the wrong inequality \eqref{turan7}. In \cite{bariczmises} the author corrected the proof of the main result of \cite{tanabe} by using the Tur\'an type inequality \eqref{turan11}.} for all $\nu\geq1/2$ and $x>0.$ This corrects the
proof of \cite[Proposition 4]{martinez}. More precisely, observe
that in view of \eqref{deltaI} the inequality \eqref{turan11} is
equivalent to
\begin{equation}\label{turan12}
y_{\nu}'(x)<1,
\end{equation}
where $y_{\nu}(x)={xI_{\nu}'(x)}/{I_{\nu}(x)}.$ In other words, to
prove \eqref{turan11} we just need to show that
$\left[y_{\nu}(x)-x\right]'<0$ for all $\nu\geq1/2$ and $x>0.$
However, the proof of this monotonicity property was given by
Gronwall \cite[p. 276]{gronwall} and is based on the inequality
\cite[p. 275]{gronwall}
\begin{equation}\label{turan13}\frac{xI_{\nu}'(x)}{I_{\nu}(x)}>x-\frac{1}{2},\end{equation} which is valid for all $x>0$ and
$\nu\geq1/2.$ We note that \eqref{turan13} can be improved as
\cite[p. 526]{segura}
\begin{equation}\label{turan14}\frac{xI_{\nu}'(x)}{I_{\nu}(x)}>\sqrt{x^2+\left(\nu-\frac{1}{2}\right)^2}-\frac{1}{2},\end{equation}
where $x>0$ and $\nu\geq1/2.$ Next, let us mention that the
inequalities \eqref{turan13} and \eqref{turan14} of Gronwall and
Segura can be improved too by using a result in the proof of
\cite[Proposition 7.2]{swatson}, due to Hartman and Watson \cite[p.
606]{swatson}. Namely, in the proof of \cite[Proposition
7.2]{swatson} it is stated that
$$r_2(x)=\left[\ln\left(\sqrt{x}I_{\nu}(x)\right)\right]'>q^{\frac{1}{2}}(x)=\sqrt{1+\frac{\nu^2-\frac{1}{4}}{x^2}},$$
that is,
\begin{equation}\label{turan15}\frac{xI_{\nu}'(x)}{I_{\nu}(x)}>\sqrt{x^2+\nu^2-\frac{1}{4}}-\frac{1}{2}\end{equation}
is valid for all $\nu\geq1/2$ and $x>0.$ It is interesting that an
alternative proof of this inequality follows from \eqref{turan11} or
\eqref{turan12}. More precisely, by using the notation
$\mu=\nu^2-1/4,$ the inequality \eqref{turan12} implies that
$xy_{\nu}'(x)<\sqrt{x^2+\mu}$ for all $\nu\geq 1/2$ and $x>0.$ On
the other hand, since $I_{\nu}$ satisfies the modified Bessel
differential equation, the function $y_{\nu}$ satisfies
\begin{equation}\label{ynu}xy_{\nu}'(x)=x^2+\nu^2-y_{\nu}^2(x)\end{equation} and consequently
$$y_{\nu}^2(x)> x^2+\nu^2-\sqrt{x^2+\mu}=\left(\sqrt{x^2+\mu}-\frac{1}{2}\right)^2,$$
or equivalently
$$\left(y_{\nu}(x)-\sqrt{x^2+\mu}+\frac{1}{2}\right)\left(y_{\nu}(x)+\sqrt{x^2+\mu}-\frac{1}{2}\right)>0,$$
which implies \eqref{turan15}. Here we used the fact that the
function $x\mapsto y_{\nu}(x)+\sqrt{x^2+\mu},$ as a sum of two
strictly increasing functions, is strictly increasing on
$(0,\infty)$ for all $\nu\geq1/2,$ and consequently
$$y_{\nu}(x)+\sqrt{x^2+\mu}>\nu+\sqrt{\mu}\geq1/2$$ for all $\nu\geq 1/2$
and $x>0.$

Now, by using the inequality \eqref{turan15} we can prove the
following theorem, which improves \eqref{turan11}.

\begin{theorem}\label{th1}
If $\nu\geq1/2$ and $x>0,$ then the next Tur\'an type inequalities
are valid
\begin{equation}\label{turan16}
\frac{\nu+\frac{1}{2}}{\nu+1}\frac{1}{\sqrt{x^2+\left(\nu+\frac{1}{2}\right)^2}}\cdot
I_{\nu}^2(x)<
I_{\nu}^2(x)-I_{\nu-1}(x)I_{\nu+1}(x)<\frac{1}{\sqrt{x^2+\nu^2-\frac{1}{4}}}\cdot
I_{\nu}^2(x).
\end{equation}
Moreover, the left-hand side of \eqref{turan16} holds true for all
$\nu\geq-1/2$ and $x>0.$ Each of the above inequalities are sharp as
$x\to\infty,$ and the left-hand side of \eqref{turan16} is sharp as
$x\to0.$
\end{theorem}

Clearly, the right-hand side of \eqref{turan16} is better than the
inequality \eqref{turan11} for all $x>0$ and $\nu>1/2.$ Moreover,
observe that the Tur\'an type inequality \eqref{turan11} is better
than the right-hand side of \eqref{turan8} for $x\geq 4(\nu+1)/3$
and $\nu\geq1/2,$ and is better than the right-hand side of
\eqref{turan9} for $x\geq \nu+1$ and $\nu\geq1/2.$ Note also that
the left-hand side of \eqref{turan16} improves the left-hand side of
\eqref{turan8} for all $\nu>-1/4$ and $x>0$ such that $x^2\leq
(4\nu+1)(4\nu+3)(\nu+1/2)^2.$ It is worth to mention here that the
relative errors of the bounds for the Tur\'anian of the modified
Bessel function of the first kind in the left-hand side of the
inequalities \eqref{turan8} and \eqref{turan9}, in inequality
\eqref{turan11} and in the right-hand side of \eqref{turan16} have
the property that tend to zero as the argument tends to infinity.
For example, the inequality \eqref{turan11} can be rewritten as
$\varphi_{\nu}(x)<1/x=r(x),$ and if we use the asymptotic formula
\cite[p. 377]{abra}
$$I_{\nu}(x)\sim \frac{e^x}{\sqrt{2\pi
x}}\left[1-\frac{4\nu^2-1}{1!(8x)}+\frac{(4\nu^2-1)(4\nu^2-9)}{2!(8x)^2}-{\dots}\right],$$
which holds for large values of $x$ and for fixed $\nu,$ one has
$\lim_{x\to\infty} \varphi_{\nu}(x)/r(x)=1$ and consequently for the
relative error we have the limit
$\lim_{x\to\infty}\left[r(x)-\varphi_{\nu}(x)\right]/\varphi_{\nu}(x)=0,$
as we required. In other words, the lower bounds in the Tur\'an type
inequalities \eqref{turan8} and \eqref{turan9}, and the upper bounds
in \eqref{turan11} and \eqref{turan16} for large values of $x$ are
quite tight. This is illustrated also on Fig. \ref{fig1}. We note
that in this figure the bounds in \eqref{turan8} are considered as
bounds for $\varphi_{\nu}(x),$ that is, they are understood in the
sense that the lower bound is
$$\frac{1}{\nu+\frac{1}{2}+\sqrt{x^2+\left(\nu+\frac{1}{2}\right)^2}},$$
while the upper bound is
$$\frac{2}{\nu+1+\sqrt{x^2+(\nu+1)^2}}.$$ The bounds in
\eqref{turan11} and \eqref{turan16} in Fig. \ref{fig1} have the same
meaning.

\begin{figure}[!ht]
   \centering
       \includegraphics[width=11cm]{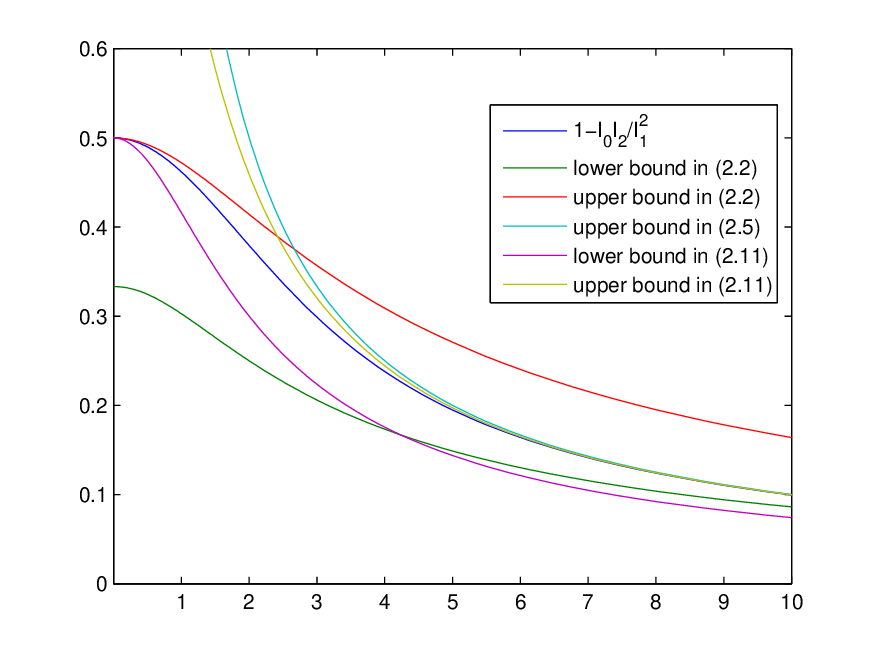}
       \caption{The graph of the function $\varphi_1$ and of the bounds in \eqref{turan8},
       \eqref{turan11} and \eqref{turan16} for $\nu=1$ on $[0,10].$}
       \label{fig1}
\end{figure}

\begin{proof}[\bf Proof of Theorem \ref{th1}] First we prove the
left-hand side of \eqref{turan16}. For this recall the
fact\footnote{For reader's convenience we note that this result of
Watson was used also by Robert \cite{robert}, Marchand and Perron
\cite{marchand0,marchand1}, Marchand and Najafabadi \cite{marchand2}
in different problems of statistics and probability.} that
\cite{swatson2} the function $x\mapsto I_{\nu+1}(x)/I_{\nu}(x)$ is
increasing and concave on $(0,\infty)$ for all $\nu\geq-1/2.$ By
using \cite[p. 77]{watson}
$$y_{\nu}(x)=\frac{xI_{\nu}'(x)}{I_{\nu}(x)}=\nu+\frac{xI_{\nu+1}(x)}{I_{\nu}(x)}$$
and the above result of Watson \cite{swatson2} we conclude that for
all $x>0$ and $\nu\geq-1/2$
\begin{equation}\label{y''}y_{\nu}''(x)=2\left[\frac{I_{\nu+1}(x)}{I_{\nu}(x)}\right]'
+x\left[\frac{I_{\nu+1}(x)}{I_{\nu}(x)}\right]''
<2\lim_{x\to0}\left[\frac{I_{\nu+1}(x)}{I_{\nu}(x)}\right]'=\frac{1}{\nu+1}.\end{equation}
Here we used the Mittag-Leffler expansion \cite[eq. 7.9.3]{erdelyi}
$$\frac{I_{\nu+1}(x)}{I_{\nu}(x)}=\sum_{n\geq1}\frac{2x}{x^2+j_{\nu,n}^2}$$
and the Rayleigh formula \cite[p. 502]{watson}
$$\sum_{n\geq1}\frac{1}{j_{\nu,n}^2}=\frac{1}{4(\nu+1)},$$
where $j_{\nu,n}$ is the $n$th positive zero of the Bessel function
$J_{\nu},$ in order to prove that
$$2\lim_{x\to0}\left[\frac{I_{\nu+1}(x)}{I_{\nu}(x)}\right]'=
2\lim_{x\to0}\sum_{n\geq1}\frac{2(j_{\nu,n}^2-x^2)}{(x^2+j_{\nu,n}^2)^2}=\sum_{n\geq1}\frac{4}{j_{\nu,n}^2}=\frac{1}{\nu+1}.$$
Now, differentiating both sides of \eqref{ynu} we obtain
\begin{equation}\label{ric2}
xy_{\nu}''(x)=2x-(2y_{\nu}(x)+1)y_{\nu}'(x),
\end{equation}
and consequently in view of \eqref{y''} we have for all $x>0$
and $\nu\geq-1/2$
$$\frac{2\nu+1}{\nu+1}x<(2y_{\nu}(x)+1)y_{\nu}'(x).$$
Combining this with the inequality \cite[p. 526]{segura}
$$y_{\nu}(x)<\sqrt{x^2+\left(\nu+\frac{1}{2}\right)^2}-\frac{1}{2}$$
we arrive at
$$\frac{1}{x}y_{\nu}'(x)>\frac{\nu+\frac{1}{2}}{\nu+1}\frac{1}{\sqrt{x^2+\left(\nu+\frac{1}{2}\right)^2}}$$
and taking into account the relation \eqref{deltaI} the proof of the
left-hand side of \eqref{turan16} is complete.

To prove the right-hand side of \eqref{turan16} we use the idea of
Gronwall \cite[p. 277]{gronwall}. Let $\mu=\nu^2-1/4.$ We prove that
the function $x\mapsto u_{\nu}(x)=\sqrt{x^2+\mu}-y_{\nu}(x)$
satisfies $u_{\nu}'(x)>0$ for all $\nu\geq1/2$ and $x>0.$ For this
observe that
$$\sqrt{x^2+\mu}=\sqrt{\mu}+\frac{x^2}{2\sqrt{\mu}}-\frac{x^4}{8\mu\sqrt{\mu}}+\dots,$$
$$\frac{xI_{\nu}'(x)}{I_{\nu}(x)}=\nu+\frac{x^2}{2(\nu+1)}-\frac{x^4}{8(\nu+1)^2(\nu+2)}+{\dots}$$
and
$$\sqrt{x^2+\mu}-y_{\nu}(x)=\sqrt{\mu}-\nu+\left[\frac{1}{\sqrt{\mu}}-\frac{1}{\nu+1}\right]\frac{x^2}{2}-
\left[\frac{1}{\mu\sqrt{\mu}}-\frac{1}{(\nu+1)^2(\nu+2)}\right]\frac{x^4}{8}+\dots,$$
which implies that for small values of $x$ the function $u_{\nu}$ is
strictly increasing. Thus the first extreme of this function, if
any, is a maximum. However, when $u_{\nu}'(x)=0,$ that is,
$$y_{\nu}'(x)=\frac{x}{\sqrt{x^2+\mu}},$$
by using \eqref{ric2} and \eqref{turan15} we have for all $x>0$ and
$\nu\geq1/2$
$$u_{\nu}''(x)=\frac{\mu}{(x^2+\mu)^{3/2}}+\frac{1-2u_{\nu}(x)}{\sqrt{x^2+\mu}}>0,$$
which is a contradiction. Consequently, the derivative of the
function $u_{\nu}$ does not vanish, and then $u_{\nu}'(x)>0$ for all
$\nu\geq1/2$ and $x>0,$ as we required. This in turn implies that
for all $x>0$ and $\nu\geq1/2$ we have
$$\frac{1}{x}y_{\nu}'(x)<\frac{1}{\sqrt{x^2+\mu}},$$
which in view of \eqref{deltaI} is equivalent to the right-hand side
of \eqref{turan16}.

Finally, let us discuss the sharpness of the inequalities. Observe
that \eqref{turan16} can be rewritten as
$$\frac{\nu+\frac{1}{2}}{\nu+1}\frac{1}{\sqrt{x^2+\left(\nu+\frac{1}{2}\right)^2}}
<\varphi_{\nu}(x)<\frac{1}{\sqrt{x^2+\nu^2-\frac{1}{4}}}.$$
Now, since \cite[p. 257]{baBAMS}
$\lim_{x\to\infty}\varphi_{\nu}(x)=0,$ each of the above
inequalities are sharp as $x\to\infty.$ Moreover, since \cite[p.
257]{baBAMS} $\lim_{x\to0}\varphi_{\nu}(x)=1/(\nu+1),$ the left-hand
side of the above Tur\'an type inequality is sharp as $x\to0.$
\end{proof}

We note that the inequality \eqref{y''} can be used also to prove
the right-hand side of the Tur\'an type inequality \eqref{turan1}
for $\nu\geq-1/2.$ More precisely, by using \eqref{y''}, for all
$x>0$ and $\nu\geq-1/2$ we get
$$\int_0^xy_{\nu}''(t)dt<\int_0^x\left(\frac{t}{\nu+1}\right)'dt,$$
that is,
$$y_{\nu}'(x)<\frac{x}{\nu+1},$$
which in view of \eqref{deltaI} is equivalent to the right-hand side
of the Tur\'an type inequality \eqref{turan1}. It is worth to
mention also here that the proof of the right-hand side of
\eqref{turan16} was motivated by Gronwall's proof \cite[p.
277]{gronwall} of the fact that the function $x\mapsto
w_{\nu}(x)=\sqrt{x^2+\nu^2}-y_{\nu}(x)$ is increasing on
$(0,\infty)$ for all $\nu>0.$ Unfortunately, Gronwall's proof is not
correct since the equation \cite[p. 277]{gronwall}
$$\frac{d^2w}{dz^2}=\frac{\nu^2}{(\nu^2+z^2)^{3/2}}+\frac{2\nu^2w}{z^2(\nu^2+z^2)^{1/2}}$$
should be rewritten as
$$\frac{d^2w}{dz^2}=\frac{\nu^2}{(\nu^2+z^2)^{3/2}}+\frac{1-2w}{(\nu^2+z^2)^{1/2}},$$
which is not necessarily positive for all $z>0$ and $\nu>0.$
Moreover, it can be proved that the function
$$x\mapsto w_{1/2}(x)=\sqrt{x^2+\frac{1}{4}}-\frac{xI_{1/2}'(x)}{I_{1/2}(x)}
=\sqrt{x^2+\frac{1}{4}}-\frac{x\cosh(x)}{\sinh(x)}+\frac{1}{2}$$
is increasing $(0,x_{1/2}]$ and decreasing on $[x_{1/2},\infty),$
where $x_{1/2}\simeq 3.577847594$ is the unique root of the equation
$w_{1/2}'(x)=0.$ Thus, Gronwall's statement that $w_{\nu}$ is
increasing on $(0,\infty)$ for all $\nu>0$ is not valid. However,
observe that to correct Gronwall's proof we would need to show that
for all $x>0$ and $\nu>0$ the following inequality is valid
$$\nu^2+(1-2w_{\nu}(x))(x^2+\nu^2)>0,$$
that is,
\begin{equation}\label{gro}y_{\nu}(x)>\sqrt{x^2+\nu^2}-\frac{x^2+2\nu^2}{2x^2+2\nu^2}.\end{equation}
By using the inequality \cite[p. 572]{paltsev}
$$-\frac{x^2}{2(x^2+\nu^2)^{3/2}}<y_{\nu}(x)-\sqrt{x^2+\nu^2}+\frac{x^2}{2(x^2+\nu^2)}$$
we can prove that \eqref{gro} is valid for all $x>0$ and $\nu\geq
1/2$ such that $x^2\leq 2\nu^3(\nu+\sqrt{\nu^2+1}).$ All the same,
we were not able to prove that the function $w_{\nu}$ is increasing
on $(0,\infty)$ for all $\nu>0.$ Computer experiments suggest that
the graph of $w_{\nu}$ intersects once the straight line $y=1/2,$
and because $w_{\nu}(x)$ tends to $1/2$ as $x$ tends to infinity,
there exists an $x_{\nu}>0$ (depending on $\nu$) such that $w_{\nu}$
is increasing on $(0,x_{\nu}]$ and decreasing on $[x_{\nu},\infty).$
Here we used the asymptotic formula \cite[p. 276]{gronwall}
$$\frac{xI_{\nu}'(x)}{I_{\nu}(x)}\sim x-\frac{1}{2}+\frac{4\nu^2-1}{8x}-\dots,$$
which holds for large values of $x$ and fixed $\nu,$ to prove that
$\lim_{x\to\infty}w_{\nu}(x)=1/2.$

Now, let us consider the function
$x\mapsto\lambda_{\nu}(x)=y_{\nu}(x)-\sqrt{x^2+(\nu+1)^2}.$ Based on
numerical experiments we believe, but are unable to prove the
following result: {\em if $\nu\geq -1/2$ and $x>0,$ then
$\lambda_{\nu}'(x)>0,$ and equivalently the Tur\'an type inequality
\begin{equation}\label{turanconj}
\frac{1}{\sqrt{x^2+(\nu+1)^2}}\cdot
I_{\nu}^2(x)<I_{\nu}^2(x)-I_{\nu-1}(x)I_{\nu+1}(x)
\end{equation}
is valid.}

Observe that, if the inequality \eqref{turanconj} would be valid,
then it would improve the left-hand side of \eqref{turan8} for $x>0$
and $\nu\geq -1/2$ such that $\nu^2+2\sqrt{x^2+(\nu+1)^2}\geq 1/2.$
Observe also that \eqref{turanconj} is better than the left-hand
side of \eqref{turan16} for all $\nu\geq-1/2$ and $x>0.$ Moreover,
\eqref{turanconj} is sharp as $x\to0$ or as $x\to \infty,$ and it
can be shown that the relative error of the bound
$1/\sqrt{x^2+(\nu+1)^2}$ in \eqref{turanconj} tends to zero as $x$
tends to infinity. On the other hand, by using the inequalities
\cite[eq. (72)]{segura}
\begin{equation}\label{tuseg}
\sqrt{x^2+(\nu+1)^2}-1<\frac{xI_{\nu}'(x)}{I_{\nu}(x)}<\sqrt{x^2+\left(\nu+\frac{1}{2}\right)^2}-\frac{1}{2},
\end{equation}
where $\nu\geq-1$ on the left-hand side and $\nu\geq -1/2$ on the
right-hand side, it is clear that $\lambda_{\nu}$ maps $(0,\infty)$
into $(-1,-1/2)$ when $\nu\geq -1/2.$ Moreover, by using the power
series representation of $y_{\nu}$ and the above asymptotic formula
for $y_{\nu},$ we obtain that $\lim_{x\to0}\lambda_{\nu}(x)=-1$ and
$\lim_{x\to\infty}\lambda_{\nu}(x)=-1/2.$ Observe that, if the inequality \eqref{turanconj} is true, then for all $\nu\geq -1/2$ and $x>0$ we have
$$\left[\sqrt{x^2+(\nu+1)^2}\right]'<y_{\nu}'(x)$$ and consequently
$$\int_0^x\left[\sqrt{t^2+(\nu+1)^2}\right]'dt<\int_0^xy_{\nu}'(t)dt,$$
which is equivalent to the left-hand side of \eqref{tuseg}. We also
mention here that the left-hand side of \eqref{tuseg} actually can
be proved also by using the properties of the function
$\lambda_{\nu}.$ More precisely, in view of the power series
representation of $y_{\nu}(x)$ and of $\sqrt{x^2+(\nu+1)^2},$ we
obtain
$$\lambda_{\nu}(x)=-1+\frac{x^4}{8(\nu+1)^3(\nu+2)}-\dots$$
and then clearly the function $\lambda_{\nu}$ is strictly increasing
and convex for small values of $x.$ Now, let $x_1$ be the smallest
positive value of $x$ for which $\lambda_{\nu}(x)$ is $-1.$ Then
$\lambda_{\nu}'(x_1)\leq0,$ that is, in view of \eqref{ynu}
$$x_1\lambda_{\nu}'(x_1)=-\lambda_{\nu}^2(x_1)-(2\nu+1)-2\lambda_{\nu}(x_1)\sqrt{x_1^2+(\nu+1)^2}-\frac{x_1^2}{\sqrt{x_1^2+(\nu+1)^2}}\leq0$$
or equivalently
$x_1^2+2(\nu+1)^2\leq2(\nu+1)\sqrt{x_1^2+(\nu+1)^2}.$ The above
inequality can be rewritten as $\sqrt{x_1^2+(\nu+1)^2}\leq\nu+1$ or
$x_1^2\leq0,$ which is a contradiction. Thus, the graph of
$\lambda_{\nu}$ does not intersect the straight line $y=-1$ and
hence $\lambda_{\nu}(x)>-1$ for all $\nu>-1$ and $x>0.$

Finally, observe that to prove \eqref{turanconj} it would enough to
show that the inequality
\begin{equation}\label{turanconj2}
\frac{xI_{\nu}'(x)}{I_{\nu}(x)}<\sqrt{x^2+(\nu+1)^2}-\frac{1}{2}\frac{x^2+2(\nu+1)^2}{x^2+(\nu+1)^2}
\end{equation}
is valid for all $x>0$ and $\nu\geq-1/2.$ Namely, since
$\lambda_{\nu}$ is increasing for small values of $x,$ the first
extreme, if any, should be a maximum. But, according to \eqref{ynu},
\eqref{ric2} and \eqref{turanconj2}, for such values of $x$ when
$\lambda_{\nu}'(x)=0,$ that is,
$$y_{\nu}'(x)=\frac{x}{\sqrt{x^2+(\nu+1)^2}},$$
we would have
$$(x^2+(\nu+1)^2)^{3/2}\lambda_{\nu}''(x)=2(x^2+(\nu+1)^2)^{3/2}-2(x^2+(\nu+1)^2)y_{\nu}(x)-(x^2+2(\nu+1)^2)>0,$$
which would be a contradiction.

\section{\bf Tur\'an type inequalities for modified Bessel functions
of the second kind} \setcounter{equation}{0}

This section is devoted to the study of Tur\'an type inequalities
for modified Bessel functions of the second kind, and our aim is to
obtain analogous results to those given in Section 3. Recently, in
order to prove \eqref{turan2}, Segura proved the next Tur\'an type
inequalities \cite[eqs. (50), (56)]{segura}
\begin{equation}\label{turan17}
-\frac{2K_{\nu}^2(x)}{\nu-1+\sqrt{x^2+(\nu-1)^2}}<K_{\nu}^2(x)-K_{\nu-1}(x)K_{\nu+1}(x)
<-\frac{K_{\nu}^2(x)}{\nu-\frac{1}{2}+\sqrt{x^2+\left(\nu-\frac{1}{2}\right)^2}},
\end{equation}
where $x>0$ and $\nu\geq 1/2.$ Observe that by changing $\nu$ with
$-\nu$ in \eqref{turan17}, and using \eqref{turan17} we obtain
\begin{equation}\label{turan18}
-\frac{2K_{\nu}^2(x)}{|\nu|-1+\sqrt{x^2+(|\nu|-1)^2}}<K_{\nu}^2(x)-K_{\nu-1}(x)K_{\nu+1}(x)<
-\frac{K_{\nu}^2(x)}{|\nu|-\frac{1}{2}+\sqrt{x^2+\left(|\nu|-\frac{1}{2}\right)^2}},
\end{equation}
where $x>0$ and $|\nu|\geq 1/2.$ These inequalities are analogous to
\eqref{turan8}. We also note that from \eqref{turan18} the following
inequalities can be obtained, which are analogous to \eqref{turan9}
\begin{equation}\label{turan19}
-\frac{2}{x+|\nu|-1}\cdot
K_{\nu}^2(x)<K_{\nu}^2(x)-K_{\nu-1}(x)K_{\nu+1}(x)<-\frac{1}{x+2|\nu|-1}\cdot
K_{\nu}^2(x),
\end{equation}
where $x>0$ and $|\nu|\geq 1/2.$ Recall that for \cite[p.
260]{baBAMS}
$$\phi_{\nu}(x)=1-\frac{K_{\nu-1}(x)K_{\nu+1}(x)}{K_{\nu}^2(x)}$$
we have $\lim_{x\to \infty}\phi_{\nu}(x)=0,$ where $\nu\geq0,$ and
$\lim_{x\to0}\phi_{\nu}(x)={1}/(1-\nu),$ provided $\nu>1.$ Thus, the
inequalities \eqref{turan18} and \eqref{turan19} are sharp as $x$
approaches infinity, while for $|\nu|>1$ the left-hand side of
\eqref{turan18} is also sharp as $x\to0.$

The next result is analogous to \eqref{turan11}.

\begin{theorem}\label{thK1}
Let $\mu=\nu^2-1/4.$ If $|\nu|\geq1/2$ and $x>0,$ then the next
Tur\'an type inequalities are valid
\begin{equation}\label{turan20}
-\frac{1}{x}\cdot K_{\nu}^2(x)\leq
K_{\nu}^2(x)-K_{\nu-1}(x)K_{\nu+1}(x)\leq-\left(1-\frac{\mu}{x^2}\right)\frac{1}{x}\cdot
K_{\nu}^2(x).
\end{equation}
Moreover, if $|\nu|<1/2$ and $x>0,$ then the above inequalities are
reversed, that is,
\begin{equation}\label{turan21}
-\left(1-\frac{\mu}{x^2}\right)\frac{1}{x}\cdot K_{\nu}^2(x)<
K_{\nu}^2(x)-K_{\nu-1}(x)K_{\nu+1}(x)<-\frac{1}{x}\cdot
K_{\nu}^2(x).
\end{equation}
In \eqref{turan20} we have equality for $\nu=1/2.$ The left-hand
side of \eqref{turan20} is sharp as $x\to0$ when
$1/2\leq|\nu|\leq1,$ while \eqref{turan21} is sharp as $x\to0$ for
all $|\nu|<1/2.$ Each of the above inequalities are sharp as
$x\to\infty.$
\end{theorem}

Observe that the left-hand side of the Tur\'an type inequality
\eqref{turan20} for $x\geq|\nu|-1>0$ is better than the left-hand
side of the inequality \eqref{turan2}, while for
$x\geq|\nu|-1\geq-1/2$ is better than the left-hand side of
\eqref{turan19}. For $x\geq 4(|\nu|-1)/3\geq-2/3$ the left-hand side
of \eqref{turan20} is also better than the left-hand side of
\eqref{turan18}. We also note that the right-hand side of
\eqref{turan20} is better than the right-hand side of \eqref{turan2}
for $x\geq\sqrt{\mu}$ and $|\nu|\geq1/2.$ The upper bound in
\eqref{turan20} is also better than the upper bound in
\eqref{turan19} when $2\alpha x\geq\mu+\sqrt{\mu^2+4\mu\alpha^2}$
for $\alpha=2|\nu|-1>0.$ Because of their different nature, it is
not easy to compare the upper bounds in \eqref{turan18} and
\eqref{turan20}. However, numerical experiments suggest that for
large values of $x$ the upper bound in \eqref{turan20} is better
than the upper bound in \eqref{turan18}. This is illustrated also on
Fig. \ref{fig2}. We note that in this figure the bounds in
\eqref{turan17} are considered as bounds for $\phi_{\nu}(x),$ that
is, they are understood in the sense that the lower bound is
$$-\frac{2}{\nu-1+\sqrt{x^2+(\nu-1)^2}},$$
while the upper bound is
$$-\frac{1}{\nu-\frac{1}{2}+\sqrt{x^2+\left(\nu-\frac{1}{2}\right)^2}}.$$ The bounds in
\eqref{turan20} in Fig. \ref{fig2} have the same meaning.

\begin{figure}[!ht]
   \centering
       \includegraphics[width=11cm]{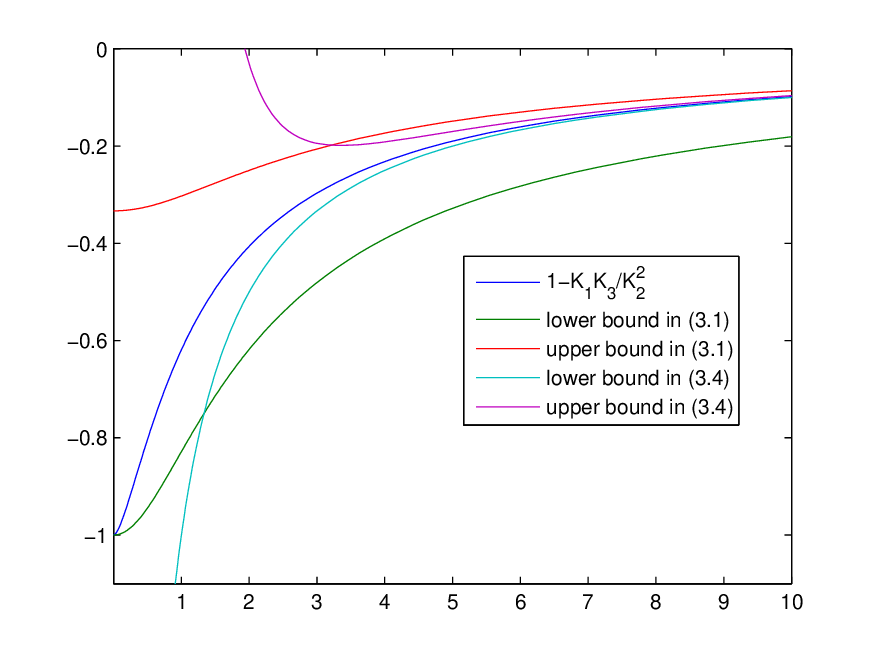}
       \caption{The graph of the function $\phi_2$ and of the bounds in \eqref{turan17} and \eqref{turan20} for $\nu=2$ on $[0,10].$}
       \label{fig2}
\end{figure}

Now, let us discuss the tightness of the bounds in \eqref{turan21}.
Having in mind from the introduction the fact that $\phi_{\nu}(x)<0$
for all $|\nu|<1/2$ and $x>0$ and in view of the notations
$$r_{\nu}(x)=-\frac{1}{x}+\frac{\mu}{x^3}\ \ \mbox{and}\ \
s(x)=-\frac{1}{x},$$ the inequality \eqref{turan21} can be rewritten
as
$$\frac{x^2}{x^2-\mu}=\frac{s(x)}{r_{\nu}(x)}<\frac{s(x)}{\phi_{\nu}(x)}<1\
\ \ \mbox{or}\ \ \
\frac{x^2-\mu}{x^2}=\frac{r_{\nu}(x)}{s(x)}>\frac{r_{\nu}(x)}{\phi_{\nu}(x)}>1.$$
These inequalities actually imply that $s(x)/\phi_{\nu}(x)$ and
$r_{\nu}(x)/\phi_{\nu}(x)$ tend to $1$ as $x\to\infty,$ and
consequently the relative errors
$$\frac{s(x)-\phi_{\nu}(x)}{\phi_{\nu}(x)}\ \ \mbox{and}\ \ \frac{r_{\nu}(x)-\phi_{\nu}(x)}{\phi_{\nu}(x)}$$
tend to $0$ as $x$ approaches infinity. These in turn imply that the
lower and upper bounds $r_{\nu}(x)$ and $s(x)$ of $\phi_{\nu}(x)$
are very tight for large values of $x.$ We note that it can be shown
in a similar way that the relative errors of the bounds in
\eqref{turan20} have the same property that tend to zero as the
argument approaches infinity. Moreover, the relative errors of the
bounds for the Tur\'anian of the modified Bessel function of the
second kind in the right-hand side of inequalities \eqref{turan18}
and \eqref{turan19} have the same property. Observe that these
properties of the bounds in \eqref{turan18}, \eqref{turan19} and
\eqref{turan20} can be proved also by using the corresponding
asymptotic relation for $K_{\nu}.$ For example, the inequality
\eqref{turan20} can be rewritten as $\phi_{\nu}(x)>-1/x=s(x),$ and
if we use the asymptotic formula \cite[p. 378]{abra}
$$K_{\nu}(x)\sim \sqrt{\frac{\pi}{2x}}e^{-x}\left[1+\frac{4\nu^2-1}{1!(8x)}+\frac{(4\nu^2-1)(4\nu^2-9)}{2!(8x)^2}+{\dots}\right],$$
which holds for large values of $x$ and for fixed $\nu,$ one has
$\lim_{x\to\infty} \phi_{\nu}(x)/s(x)=1$ and consequently for the
relative error we have
$\lim_{x\to\infty}\left[s(x)-\phi_{\nu}(x)\right]/\phi_{\nu}(x)=0,$
as we required. In other words, the upper bounds in the Tur\'an type
inequalities \eqref{turan18} and \eqref{turan19}, and also the lower
and upper bounds in \eqref{turan20} for large values of $x$ are
quite tight.

\begin{proof}[\bf Proof of Theorem \ref{thK1}]
First recall that the function $\nu\mapsto K_{\nu}(x)$ is even, that
is, we have \cite[p. 79]{watson} $K_{-\nu}(x)=K_{\nu}(x).$ Because
of this, without loss of generality, it is enough to prove the
inequality \eqref{turan20} for $\nu\geq1/2$ and the inequality
\eqref{turan21} for $0\leq\nu<1/2.$ Recall also that by using
Ismail's formula \cite[p. 583]{ismail}, \cite[p. 356]{ismail2}
$$\frac{K_{\nu-1}(\sqrt{x})}{\sqrt{x}K_{\nu}(\sqrt{x})}=
\frac{4}{\pi^2}\int_0^{\infty}\frac{\gamma_{\nu}(t)dt}{x+t^2},\ \ \
\mbox{where}\ \ \
\gamma_{\nu}(t)=\frac{t^{-1}}{J_{\nu}^2(t)+Y_{\nu}^2(t)},$$ where
$x>0,$ $\nu\geq0$ and $J_{\nu}$ and $Y_{\nu}$ stand for the Bessel
function of the first and second kinds, it can be shown that
\cite[p. 260]{baBAMS}\footnote{It should be mentioned here that in
\cite[p. 260]{baBAMS} the expressions
$$\phi_{\nu}(x)=-\frac{4}{\pi^2}\int_0^{\infty}\frac{(x^2+t^2+1)\gamma(t)dt}{(x^2+t^2)^2},\ \
\phi'_{\nu}(x)=\frac{8}{\pi^2}\int_0^{\infty}\frac{x(x^2+t^2+2)\gamma(t)dt}{(x^2+t^2)^3}$$
are not correct and should be rewritten as
$$\phi_{\nu}(x)=-\frac{8}{\pi^2}\int_0^{\infty}\frac{t^2\gamma(t)dt}{(x^2+t^2)^2},\ \
\phi'_{\nu}(x)=\frac{32}{\pi^2}\int_0^{\infty}\frac{xt^2\gamma(t)dt}{(x^2+t^2)^3}.$$
See also \cite{bapo} for more details.}
\begin{equation}\label{phi}\phi_{\nu}(x)=\frac{1}{x}\left[\frac{xK_{\nu}'(x)}{K_{\nu}(x)}\right]'
=-\frac{8}{\pi^2}\int_0^{\infty}\frac{t^2\gamma_{\nu}(t)dt}{(x^2+t^2)^2}.\end{equation}
On the other hand, it is known that \cite[p. 446]{watson} the
function $t\mapsto 1/\gamma_{\nu}(t)$ is decreasing on $(0,\infty)$
for all $\nu>1/2$ and is increasing on $(0,\infty)$ for all
$0\leq\nu<1/2.$ Consequently, we obtain that $\gamma_{\nu}(t)<\pi/2$
for all $t>0$ and $\nu>1/2.$ Moreover, $\gamma_{\nu}(t)>\pi/2$ for
all $t>0$ and $0\leq\nu<1/2.$ Thus, we have
$$\phi_{\nu}(x)>-\frac{4}{\pi}\int_0^{\infty}\frac{t^2dt}{(x^2+t^2)^2}=-\frac{1}{x},$$
where $\nu>1/2$ and $x>0.$ The same proof works in the case
$0\leq\nu<1/2.$ The only difference is that the above inequality is
reversed. Now, by using for $\nu=1/2$ the relations \cite[p.
79]{watson}
$$K_{\nu+1}(x)-K_{\nu-1}(x)=\frac{2\nu}{x}K_{\nu}(x),\ \ K_{\nu}(x)=K_{-\nu}(x),$$
we obtain $\phi_{1/2}(x)=-1/x.$ This completes the proof of the
left-hand side of \eqref{turan20} and of the right-hand side of
\eqref{turan21}. We note that there is another proof for these
results. Namely, in view of the Nicholson formula \cite{watson}
$$J_{\nu}^2(t)+Y_{\nu}^2(t)=\frac{8}{\pi^2}\int_0^{\infty}K_0(2t\sinh s)\cosh(2\nu s)ds,$$
the function $\nu\mapsto\gamma_{\nu}(t)$ is decreasing on
$[0,\infty)$ for all $t>0$ fixed. This in turn implies that the
function $\nu\mapsto\phi_{\nu}(x)$ is increasing on $[0,\infty)$ for
all $x>0$ fixed. Consequently, $\phi_{\nu}(x)\geq
\phi_{1/2}(x)=-1/x$ for all $x>0$ and $\nu\geq1/2,$ and
$\phi_{\nu}(x)<\phi_{1/2}(x)=-1/x$ for all $x>0$ and $0\leq\nu<1/2.$

Now, let us focus on the right-hand side of \eqref{turan20} and on
the left-hand side of \eqref{turan21}. Observe that the inequality
\cite[eq. (4.6)]{hartman}
\begin{equation}\label{hartman}t\left(1-\frac{\mu}{t^2}\right)\left[J_{\nu}^2(t)+Y_{\nu}^2(t)\right]<\frac{2}{\pi},\end{equation}
where $t>0$ and $\nu>1/2,$ is equivalent to
\begin{equation}\label{gamma}\gamma_{\nu}(t)>\left(1-\frac{\mu}{t^2}\right)\frac{\pi}{2}.\end{equation}
Since
\begin{equation}\label{pitagoras}t\left[J_{1/2}^2(t)+Y_{1/2}^2(t)\right]=t\left[\frac{2}{\pi
t}\sin^2t+\frac{2}{\pi t}\cos^2t\right]=\frac{2}{\pi},\end{equation}
for $\nu=1/2$ in inequalities \eqref{hartman} and \eqref{gamma} we
have equality. These in turn imply that for all $x>0$ and
$\nu\geq1/2$ we have
$$\phi_{\nu}(x)\leq-\frac{4}{\pi}\int_0^{\infty}\frac{t^2dt}{(x^2+t^2)^2}+
\frac{4\mu}{\pi}\int_0^{\infty}\frac{dt}{(x^2+t^2)^2}=-\frac{1}{x}+\frac{\mu}{x^3},$$
with equality when $\nu=1/2,$ that is, $\mu=0.$ The same proof works
in the case $0\leq\nu<1/2.$ The only difference is that the
inequality \eqref{hartman} is reversed, according to \cite[eq.
(4.7)]{hartman}, and then \eqref{gamma} is reversed too.

Finally, let us discuss the sharpness of inequalities. Observe that
\eqref{turan20} and \eqref{turan21} can be rewritten as
$$-\frac{1}{x}\leq \phi_{\nu}(x)\leq-\frac{1}{x}+\frac{\mu}{x^3}\ \ \ \mbox{and}\ \ \ -\frac{1}{x}+\frac{\mu}{x^3}<\phi_{\nu}(x)<-\frac{1}{x}.$$
Since for all $\nu\geq0$ we have \cite[p. 260]{baBAMS} $\lim_{x\to
\infty}\phi_{\nu}(x)=0,$ clearly both of the above inequalities are
sharp as $x\to\infty.$ Moreover, because \cite[p. 260]{baBAMS}
$\lim_{x\to0}\phi_{\nu}(x)={1}/(1-\nu),$ provided $\nu>1,$ the
inequality \eqref{turan20} is not sharp as $x\to0.$ But using the
asymptotic relation \cite[p. 375]{watson} $2K_{\nu}(x)\sim
\Gamma(\nu)(x/2)^{-\nu}$ as $x\to0$ and $\nu>0,$ we obtain that for
$\nu\in(0,1)$
$$\phi_{\nu}(x)\sim 1-\frac{\Gamma(1-\nu)\Gamma(1+\nu)}{\Gamma^2(\nu)}
\left(\frac{x}{2}\right)^{2\nu-2},$$ and then we have
$\lim_{x\to0}\phi_{\nu}(x)=-\infty.$ Combining the above asymptotic
relation with \cite[p. 375]{watson} $K_0(x)\sim -\ln x,$ we obtain
$\phi_{1}(x)\sim 1+\ln x,$ and thus
$\lim_{x\to0}\phi_{1}(x)=-\infty.$ These show that the left-hand
side of the inequality \eqref{turan20} is sharp as $x\to0$ when
$1/2\leq|\nu|\leq1,$ while \eqref{turan21} is sharp as $x\to0$ for
all $|\nu|<1/2.$
\end{proof}

We note that in the proof of \cite[Proposition 7.2]{swatson} it is
stated that
$$r_1(x)=\left[\ln\left(\sqrt{x}K_{\nu}(x)\right)\right]'>-q^{\frac{1}{2}}(x)=-\sqrt{1+\frac{\nu^2-\frac{1}{4}}{x^2}},$$
that is,
\begin{equation}\label{turan22}z_{\nu}(x)=\frac{xK_{\nu}'(x)}{K_{\nu}(x)}>-\sqrt{x^2+\nu^2-\frac{1}{4}}-\frac{1}{2}\end{equation}
is valid for all $\nu\geq1/2$ and $x>0.$ Observe that since
$K_{1/2}(x)=\sqrt{\pi/(2x)}e^{-x},$ we have $z_{1/2}(x)=-x-1/2$ and
in \eqref{turan22} for $\nu=1/2$ we have equality, and by using the
symmetry with respect to $\nu,$ we conclude that \eqref{turan22} is
valid for all $|\nu|\geq1/2$ and $x>0.$ Moreover, it is worth to
note here that the left-hand side of the Tur\'an type inequality
\eqref{turan20} implies the inequality \eqref{turan22}. More
precisely, in view of \eqref{deltaK} the left-hand side of
\eqref{turan20} is equivalent to $z_{\nu}'(x)\geq-1$ for all
$|\nu|\geq1/2$ and $x>0.$ This implies that $xz_{\nu}'(x)\geq
-\sqrt{x^2+\mu}$ for all $\mu=\nu^2-1/4\geq0$ and $x>0.$ On the
other hand, since $K_{\nu}$ satisfies the modified Bessel
differential equation, the function $z_{\nu}$ satisfies
\begin{equation}\label{znu}xz_{\nu}'(x)=x^2+\nu^2-z_{\nu}^2(x)\end{equation} and consequently
$$z_{\nu}^2(x)\leq x^2+\nu^2+\sqrt{x^2+\mu}=\left(\sqrt{x^2+\mu}+\frac{1}{2}\right)^2,$$
which implies \eqref{turan22}.

Similar bounds to \eqref{turan22} for the logarithmic derivative of
$K_{\nu}$ were given also in \cite{paltsev,segura} for $\nu\geq0$
and $x>0$. For $\nu\geq1/2$ the inequality \eqref{turan22} improves
\cite[eq. (74)]{segura}
$$\frac{xK_{\nu}'(x)}{K_{\nu}(x)}>-\sqrt{x^2+\left(\nu+\frac{1}{2}\right)^2}-\frac{1}{2},$$
and also improves \cite[eq. (22)]{paltsev}
\begin{equation}\label{paltsev}\frac{xK_{\nu}'(x)}{K_{\nu}(x)}>-\sqrt{x^2+\nu^2}-\frac{1}{2}.\end{equation}
In addition, for $\nu\geq1/2$ and $x^2\geq 3\nu^2-4\nu+5/4$ the
inequality \eqref{turan22} improves \cite[eq. (75)]{segura}
$$\frac{xK_{\nu}'(x)}{K_{\nu}(x)}>-\sqrt{x^2+(\nu-1)^2}-1.$$

Now, we are going to improve the left-hand side of the inequality
\eqref{turan21}. Observe that \eqref{gamma2} improves the reversed
form of \eqref{gamma} and hence the left-hand side of
\eqref{turan23} improves the left-hand side of \eqref{turan21}. We
note that the expression on the left-hand side of \eqref{turan23}
divided by $K_{\nu}^2(x)$ provides a tight lower bound for
$\phi_{\nu}(x),$ its relative error tends to zero as $x$ approaches
infinity.

\begin{theorem}
If $\mu=\nu^2-1/4\leq0$ and $x>\sqrt{-\mu},$ then the next Tur\'an
type inequality is valid
\begin{equation}\label{turan23}
-\frac{4}{\pi}\left[\frac{\arccos\left(\frac{\sqrt{-\mu}}{x}\right)}{2\sqrt{x^2+\mu}}+\frac{\sqrt{-\mu}}{2x^2}\right]
\cdot K_{\nu}^2(x)\leq K_{\nu}^2(x)-K_{\nu-1}(x)K_{\nu+1}(x).
\end{equation}
In \eqref{turan23} we have equality for $\nu=1/2.$  The above
inequality is sharp as $x\to\infty.$
\end{theorem}

\begin{proof}[\bf Proof]
In what follows, without loss of generality, we assume that
$0\leq\nu\leq1/2.$ Consider the inequality \cite[eq.
(4.10)]{hartman}
$$\sqrt{t^2-\mu}\left[J_{\nu}^2(t)+Y_{\nu}^2(t)\right]>\frac{2}{\pi},$$
where $t>0$ and $0\leq\nu<1/2.$ Observe that by using
\eqref{pitagoras}, for $\nu=1/2$ in the above inequality we have
equality. Consequently, for all $t>0$ and $0\leq\nu\leq1/2$ we
obtain
\begin{equation}\label{gamma2}
\gamma_{\nu}(t)\leq\sqrt{1-\frac{\mu}{t^2}}\frac{\pi}{2}
\end{equation}
and using \eqref{phi} we conclude that
$$\phi_{\nu}(x)\geq-\frac{4}{\pi}\int_0^{\infty}\frac{t^2\sqrt{1-\frac{\mu}{t^2}}dt}{(x^2+t^2)^2}=
-\frac{4}{\pi}\left[\frac{\arccos\left(\frac{\sqrt{-\mu}}{x}\right)}{2\sqrt{x^2+\mu}}+\frac{\sqrt{-\mu}}{2x^2}\right].$$\end{proof}

Next, we improve the right-hand side of \eqref{turan20}.

\begin{theorem}\label{thK3}
If $|\nu|\geq1/2$ and $x>0,$ then the following Tur\'an type
inequality holds
\begin{equation}\label{turan24}K_{\nu}^2(x)-K_{\nu-1}(x)K_{\nu+1}(x)\leq-\frac{1}{\sqrt{x^2+\nu^2-\frac{1}{4}}}\cdot K_{\nu}^2(x).\end{equation}
In \eqref{turan24} we have equality for $\nu=1/2.$ This inequality
is sharp as $x\to\infty.$
\end{theorem}

Observe that \eqref{turan24} improves the right-hand side of
\eqref{turan17} for all $\nu\geq3/2$ and $x>0,$ and it is clearly
better than the right-hand side of \eqref{turan20} for all
$|\nu|>1/2$ and $x>0.$ Moreover, by using the asymptotic formula for
$K_{\nu}(x)$ for large $x,$ as above, it can be proved that the
relative error of the bound in \eqref{turan24} has the property that
tends to zero as $x$ tends to infinity. Finally, observe that by
using \eqref{deltaK}, the inequality \eqref{turan24} can be
rewritten as
$$\left[\frac{xK_{\nu}'(x)}{K_{\nu}(x)}\right]'\leq-\left[\sqrt{x^2+\mu}\right]',$$
which implies
$$\int_0^x\left[\frac{tK_{\nu}'(t)}{K_{\nu}(t)}\right]'dt\leq-\int_0^x\left[\sqrt{t^2+\mu}\right]'dt,$$
that is,
$$\frac{xK_{\nu}'(x)}{K_{\nu}(x)}\leq-\sqrt{x^2+\mu}+\sqrt{\mu}-\nu,$$
where $\mu=\nu^2-1/4\geq0$ and $x>0.$ Observe that for all
$|\nu|\geq1/2$ and $x>0$ this inequality is better than
\eqref{turan4}, however, it is weaker than the left-hand side of the
inequality\footnote{We note that in the left-hand side of \cite[eq.
(75)]{segura} it is assumed that $\nu\geq1.$ However, because of
\cite[eq. (30)]{segura}, we can suppose that $\nu\geq1/2$ in the
above inequality.} \cite[eq. (75)]{segura}
\begin{equation}\label{segura}\frac{xK_{\nu}'(x)}{K_{\nu}(x)}\leq
-\sqrt{x^2+\left(\nu-\frac{1}{2}\right)^2}-\frac{1}{2}.\end{equation}

\begin{proof}[\bf Proof of Theorem \ref{thK3}]
Since $\phi_{1/2}(x)=-1/x,$ in \eqref{turan24} for $\nu=1/2$ we have
equality. Thus, without loss of generality, we suppose that
$\nu>1/2.$ Because of \eqref{deltaK} to prove \eqref{turan24} we
need to show that the function $x\mapsto
q_{\nu}(x)=z_{\nu}(x)+\sqrt{x^2+\mu},$ where $\mu=\nu^2-1/4,$
satisfies $q_{\nu}'(x)<0$ for all $\nu>1/2$ and $x>0.$ By using
\eqref{segura} it results that
$$q_{\nu}(x)\leq \sqrt{x^2+\mu}-\sqrt{x^2+\left(\nu-\frac{1}{2}\right)^2}-\frac{1}{2}\leq\sqrt{\mu}-\nu=\lim_{x\to0}q_{\nu}(x)$$ for
all $\nu\geq1/2$ and $x>0.$ On the other hand, according to
\eqref{turan22} we have $q_{\nu}(x)>-1/2$ for all $\nu>1/2$ and
$x>0.$ Moreover, in view of the asymptotic relation \cite[eq.
(20)]{paltsev}
$$\frac{xK_{\nu}'(x)}{K_{\nu}(x)}\sim -x-\frac{1}{2}-\frac{4\nu^2-1}{8x}+\frac{4\nu^2-1}{8x^2}-\dots,$$
which holds for large values of $x$ and fixed $\nu,$ we obtain
$\lim_{x\to\infty}q_{\nu}(x)=-1/2.$ In other words, for all $x>0$
and $\nu>1/2$ we have
$$\lim_{x\to0}q_{\nu}(x)>q_{\nu}(x)>\lim_{x\to\infty}q_{\nu}(x).$$
It is also clear that by using \eqref{phi} we have
$\lim_{x\to0}q_{\nu}'(x)=0.$ Thus, for small values of $x$ the
function $q_{\nu}$ is decreasing. Now, suppose that $q_{\nu}'(x)$
vanish for some $x>0.$ Since $\lim_{x\to\infty}q_{\nu}(x)=-1/2$ and
$\lim_{x\to0}q_{\nu}(x)>-1/2$ for $\nu>1/2$ it follows that
$q_{\nu}'(x)$ will vanish at least one more time, and then the
second extreme, if any, should be a local maximum. However, for $x$
such that $q_{\nu}'(x)=0,$ that is,
$$z_{\nu}'(x)=-\frac{x}{\sqrt{x^2+\mu}}$$
we have
$$q_{\nu}''(x)=\frac{\mu}{(x^2+\mu)^{3/2}}+\frac{2q_{\nu}(x)+1}{\sqrt{x^2+\mu}}>0,$$
according to \eqref{turan22} and the relation
$xz_{\nu}''(x)=2x-(2z_{\nu}(x)+1)z_{\nu}'(x),$ which follows from
\eqref{znu}. But, this is a contradiction. Consequently, the
derivative of $q_{\nu}$ does not vanish on $(0,\infty)$ and then
$q_{\nu}'(x)<0$ for all $\nu>1/2$ and $x>0,$ as we required.
\end{proof}

We note that following the steps of the above proof it can be proved
that, if $\nu\in\mathbb{R}$ and $x>0,$ then
\begin{equation}\label{turan25}K_{\nu}^2(x)-K_{\nu-1}(x)K_{\nu+1}(x)\leq-\frac{1}{\sqrt{x^2+\nu^2}}\cdot K_{\nu}^2(x).\end{equation}
More precisely, if we suppose that $\nu>0$ and consider the function
$x\mapsto t_{\nu}(x)=z_{\nu}(x)+\sqrt{x^2+\nu^2},$ then according to
\eqref{turan4} and \eqref{paltsev} we have
$$0=\lim_{x\to0}t_{\nu}(x)>t_{\nu}(x)>\lim_{x\to\infty}t_{\nu}(x)=-\frac{1}{2}.$$
Moreover, $\lim_{x\to0}t_{\nu}'(x)=0.$ Thus, for small values of $x$
the function $t_{\nu}$ is decreasing. Now, if we suppose that
$t_{\nu}'(x)$ vanish for some $x>0,$ then $t_{\nu}'(x)$ will vanish
at least one more time, and then the second extreme, if any, should
be a local maximum. However, for $x$ such that $t_{\nu}'(x)=0,$ that
is,
$$z_{\nu}'(x)=-\frac{x}{\sqrt{x^2+\nu^2}}$$
we have
$$t_{\nu}''(x)=\frac{\nu^2}{(x^2+\nu^2)^{3/2}}+\frac{2t_{\nu}(x)+1}{\sqrt{x^2+\nu^2}}>0,$$
which is a contradiction. Consequently, the derivative of $t_{\nu}$
does not vanish on $(0,\infty)$ and then $t_{\nu}'(x)<0$ for all
$\nu>1/2$ and $x>0.$ Note however, that the Tur\'an type inequality
\eqref{turan25} is weaker than \eqref{turan24} for $|\nu|\geq 1/2$
and $x>0,$ and it is also weaker than the right-hand side of the
inequality \eqref{turan21} for $|\nu|<1/2$ and $x>0.$ All the same,
this result can be used to prove \eqref{turan4}. Namely, in view of
\eqref{deltaK} the inequality \eqref{turan25} is equivalent to
$$\left[\frac{xK_{\nu}'(x)}{K_{\nu}(x)}\right]'<-\left[\sqrt{x^2+\nu^2}\right]',$$
which implies
$$\int_0^x\left[\frac{tK_{\nu}'(t)}{K_{\nu}(t)}\right]'dt<-\int_0^x\left[\sqrt{t^2+\nu^2}\right]'dt,$$
that is, the inequality \eqref{turan4}.

\section{\bf Inequalities for product of modified Bessel functions
of the first and second kind} \setcounter{equation}{0}

In this section we present some applications of the main results of
Section 2 and 3. By definition a function
$f:[a,b]\subseteq\mathbb{R}\to(0,\infty)$ is log-convex if $\ln f$
is convex, i.e. if for all $x,y\in[a,b]$ and $\lambda\in[0,1]$ we
have
$$f(\lambda x+(1-\lambda)y)\leq \left[f(x)\right]^{\lambda}\left[f(y)\right]^{1-\lambda}.$$
Similarly, a function $g:[a,b]\subseteq(0,\infty)\to(0,\infty)$ is
said to be geometrically (or multiplicatively) convex if $g$ is
convex with respect to the geometric mean, i.e. if for all
$x,y\in[a,b]$ and $\lambda\in[0,1]$ we have
$$g\left(x^\lambda y^{1-\lambda}\right)\leq \left[g(x)\right]^{\lambda}\left[g(y)\right]^{1-\lambda}.$$
We note that if the functions $f$ and $g$ are differentiable then
$f$ is (strictly) log-convex if and only if the function $x\mapsto
f'(x)/f(x)$ is (strictly) increasing on $[a,b]$, while $g$ is
(strictly) geometrically convex if and only if the function
$x\mapsto xg'(x)/g(x)$ is (strictly) increasing on $[a,b].$ A
similar definition and characterization of differentiable (strictly)
log-concave and (strictly) geometrically concave functions also
holds. Observe that the left-hand side of \eqref{turan1} together
with \eqref{deltaI}, and the right-hand side of \eqref{turan2}
together with \eqref{deltaK} imply that $I_{\nu}$ is strictly
geometrically convex on $(0,\infty)$ for all $\nu>-1,$ while
$K_{\nu}$ is strictly geometrically concave on $(0,\infty)$ for all
$\nu\in \mathbb{R},$ respectively. Moreover, summing up the
corresponding parts of the right-hand sides of Tur\'an type
inequalities \eqref{turan16} and \eqref{turan24} and taking into
account the relations \eqref{deltaI} and \eqref{deltaK} we obtain
\begin{equation}\label{produ}\left[\frac{xP_{\nu}'(x)}{P_{\nu}(x)}\right]'=
\left[\frac{xI_{\nu}'(x)}{I_{\nu}(x)}\right]'+\left[\frac{xK_{\nu}'(x)}{K_{\nu}(x)}\right]'<0\end{equation}
for all $\nu\geq1/2$ and $x>0.$

Consequently, the following result is valid.

\begin{corollary}
If $\nu\geq1/2,$ then the function $P_{\nu}$ is strictly
geometrically concave on $(0,\infty).$ In particular, for all
$x,y>0$ and $\nu\geq1/2$ we have
$$P_{\nu}(\sqrt{xy})>\sqrt{P_{\nu}(x)P_{\nu}(y)}.$$
\end{corollary}

It is also important to note here that since for
$\omega_{\nu}(x)=xP_{\nu}(x)=xI_{\nu}(x)K_{\nu}(x)$ we have
$$\frac{x\omega_{\nu}'(x)}{\omega_{\nu}(x)}=1+\frac{xP_{\nu}'(x)}{P_{\nu}(x)},$$
the above result implies that the function $\omega_{\nu}$ is also
strictly geometrically concave on $(0,\infty)$ for all $\nu\geq1/2.$
On the other hand, since the function $2\omega_{\nu}$ is a
continuous cumulative distribution function, according to
\cite[Proposition 7.2]{swatson}, it follows that the $\omega_{\nu}$
is strictly log-concave on $(0,\infty)$ for all $\nu\geq1/2.$ This
results is similar to the result of Hartman \cite{hartman2}, who
proved that $\omega_{\nu}$ is strictly concave on $(0,\infty)$ for
all $\nu>1/2.$ Since $x\mapsto 2\omega_{1/2}(x)=1-e^{-2x}$ is
strictly concave on $(0,\infty),$ we conclude that in fact the
function $\omega_{\nu}$ is strictly concave, and hence strictly
log-concave on $(0,\infty)$ for all $\nu\geq 1/2.$ We also mention
here that recently in \cite{bageo} it was shown that surprisingly
the most common continuous univariate distributions, like the
standard normal, standard log-normal (or Gibrat), Student's $t$,
Weibull (or Rosin-Rammler), Kumaraswamy, Fisher-Snedecor's $F$,
gamma and Sichel (or generalized inverse Gaussian) distributions,
have the property that their probability density functions are
geometrically concave and consequently their cumulative distribution
functions and survival functions are also geometrically concave.
Taking into account the above discussion, the distribution of which
cumulative distribution function $2\omega_{\nu}$ was considered by
Hartman and Watson \cite[Proposition 7.2]{swatson} belongs also to
the class of geometrically concave univariate distributions.

Observe that if we combine the inequality \eqref{produ} with the
Wronskian recurrence relation
$$\frac{xI_{\nu}'(x)}{I_{\nu}(x)}-\frac{xK_{\nu}'(x)}{K_{\nu}(x)}=\frac{1}{P_{\nu}(x)},$$
then we obtain the following chain of inequalities
$$2\left[\frac{xK_{\nu}'(x)}{K_{\nu}(x)}\right]'<\frac{P_{\nu}'(x)}{P_{\nu}^2(x)}<-2\left[\frac{xI_{\nu}'(x)}{I_{\nu}(x)}\right]',$$
where $\nu\geq 1/2$ and $x>0.$ In other words, by using
\eqref{deltaI} and \eqref{deltaK}, for $\nu\geq 1/2$ and $x>0$ the
left-hand side of the Tur\'an type inequality \eqref{turan1} implies
the fact that the product of modified Bessel functions of the first
and second kind is strictly decreasing, which implies the right-hand
side of the Tur\'an type inequality \eqref{turan2}. Thus, when
$\nu\geq 1/2$ the left-hand side of the Tur\'an type inequality
\eqref{turan1} is stronger than the right-hand side of
\eqref{turan2}.

Now, let us show some Tur\'an type inequalities for the product of
modified Bessel functions of the first and second kind.

\begin{corollary}\label{corpro}
Let $\mu=\nu^2-1/4.$ If $\nu\geq1/2$ and $x>0,$ then the next
Tur\'an type inequality is valid
\begin{equation}\label{turan26}\frac{\left[x-\left(\nu+\frac{1}{2}\right)-\sqrt{x^2+\left(\nu+\frac{1}{2}\right)^2}\right]\sqrt{x^2+\mu}+x}
{x\sqrt{x^2+\mu}\left[\nu+\frac{1}{2}+\sqrt{x^2+\left(\nu+\frac{1}{2}\right)^2}\right]}\cdot
P_{\nu}^2(x)<
P_{\nu}^2(x)-P_{\nu-1}(x)P_{\nu+1}(x)<\frac{P_{\nu}^2(x)}{x\sqrt{x^2+\mu}}.\end{equation}
Both of the inequalities are sharp as $x\to\infty.$
\end{corollary}

We note that by using the inequalities \eqref{turan1} and
\eqref{turan2} clearly we can deduce some Tur\'an type inequalities
for the product of modified Bessel functions. However, the
inequalities obtained in this way are far from being sharp. Now, the
bounds in \eqref{turan26} are sharp for large values of $x$ and it
can be shown by using the asymptotic formula for the product of
modified Bessel functions of the first and second kind that the
relative errors of the bounds in \eqref{turan27} tend to zero as $x$
approaches infinity. Thus, the bounds in \eqref{turan27} are tight
for large values of $x.$ This is illustrated in Fig. \ref{fig3}.

\begin{figure}[!ht]
   \centering
       \includegraphics[width=11cm]{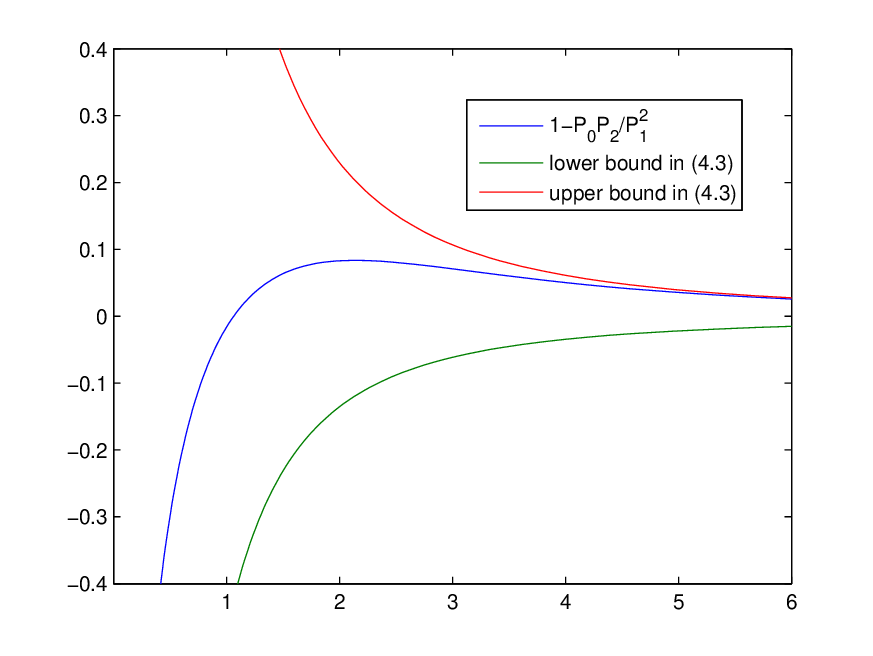}
       \caption{The graph of the function $x\mapsto 1-P_0(x)P_2(x)/P_1^2(x)$ and of the bounds in \eqref{turan27} for $\nu=1$ on $[0,6].$}
       \label{fig3}
\end{figure}

\begin{proof}[\bf Proof of Corollary \ref{corpro}] By using the left-hand sides of \eqref{turan8} and \eqref{turan20}
we obtain
$$\varphi_{\nu}(x)+\phi_{\nu}(x)>\frac{1}{\nu+\frac{1}{2}+\sqrt{x^2+\left(\nu+\frac{1}{2}\right)^2}}-\frac{1}{x}.$$
Similarly, by using the left-hand side of \eqref{turan8} and the
right-hand side of \eqref{turan24}, one has
$$-\varphi_{\nu}(x)\phi_{\nu}(x)>\frac{1}{x\sqrt{x^2+\mu}}\frac{1}{\nu+\frac{1}{2}+\sqrt{x^2+\left(\nu+\frac{1}{2}\right)^2}}.$$
On the other hand
$$1-\frac{P_{\nu-1}(x)P_{\nu+1}(x)}{P_{\nu}^2(x)}=\varphi_{\nu}(x)+\phi_{\nu}(x)-\varphi_{\nu}(x)\phi_{\nu}(x),$$
and summing up the the corresponding parts of the above inequalities
the proof of the left-hand side of \eqref{turan26} is done. Now, by
using the right-hand sides of \eqref{turan16} and \eqref{turan24} we
obtain that $$\varphi_{\nu}(x)+\phi_{\nu}(x)<0$$ for all $\nu\geq
1/2$ and $x>0.$ Similarly, by using the right-hand side of
\eqref{turan16} and the left-hand side of \eqref{turan20}, we get
$$-\varphi_{\nu}(x)\phi_{\nu}(x)<\frac{1}{x\sqrt{x^2+\mu}}.$$ These
inequalities imply the right-hand side of \eqref{turan26}. Now, let
us focus on the sharpness when $x\to\infty.$ Clearly \eqref{turan26}
can be rewritten as
\begin{equation}\label{turan27}\frac{\left[x-\left(\nu+\frac{1}{2}\right)-\sqrt{x^2+\left(\nu+\frac{1}{2}\right)^2}\right]\sqrt{x^2+\mu}+x}
{x\sqrt{x^2+\mu}\left[\nu+\frac{1}{2}+\sqrt{x^2+\left(\nu+\frac{1}{2}\right)^2}\right]}<
1-\frac{P_{\nu-1}(x)P_{\nu+1}(x)}{P_{\nu}^2(x)}<\frac{1}{x\sqrt{x^2+\mu}}.\end{equation}
In view of the asymptotic relation \cite[p. 378]{abra}
$$I_{\nu}(x)K_{\nu}(x)\sim \frac{1}{2x}\left[1-\frac{1}{2}\frac{4\nu^2-1}{(2x)^2}+\frac{1\cdot3}{2\cdot4}\frac{(4\nu^2-1)(4\nu^2-9)}{(2x)^4}+{\dots}\right],$$
which holds for large values of $x$ and for fixed $\nu,$ one has
$$\lim_{x\to\infty}\left[1-\frac{P_{\nu-1}(x)P_{\nu+1}(x)}{P_{\nu}^2(x)}\right]=0,$$
and thus the lower and upper bounds in \eqref{turan27} are sharp as
$x$ approaches infinity.
\end{proof}

\end{document}